\documentclass[10pt]{article}
\usepackage{spconf}

\usepackage{verbatim}
\usepackage{amsfonts,amsmath,mathrsfs,amssymb,amsbsy}
\usepackage[final]{graphicx}
\usepackage{times,cite}
\usepackage{caption}
\usepackage{subcaption}

\usepackage{enumitem,kantlipsum}

\long\def\symbolfootnote[#1]#2{\begingroup%
\def\thefootnote{\fnsymbol{footnote}}\footnote[#1]{#2}\endgroup}

\usepackage{verbatim}
\usepackage[]{float,latexsym}
\usepackage{amsfonts,amsmath,mathrsfs,amssymb,amsbsy}
\usepackage{url}
\usepackage{amsthm}

\newtheorem{theorem}{Theorem}[section]

\newtheorem{lemma}{Lemma}
\newtheorem{definition}{Definition}

\newcommand{\Prob}{\mathsf{P}}
\newcommand{\Expect}{\mathsf{E}}

\usepackage{color}
\definecolor{lightblue}{rgb}{.7, .8, 1}
\definecolor{lightgreen}{rgb}{.6, 1, .6}
\usepackage{color}
\definecolor{brown}{rgb}{1,0.38,0.03}

\definecolor{OliveGreen}{rgb}{.2,0.6,0.2}
\definecolor{BrickRed}{rgb}{.7,0.2,0.2}

\newcommand{\ignore}[1]{} %%% {} empty inside
\long\def\symbolfootnote[#1]#2{\begingroup%
\def\thefootnote{\fnsymbol{footnote}}\footnote[#1]{#2}\endgroup}

\DeclareMathOperator*{\esssup}{ess\,sup}

\newcommand{\bsp}{\begin{split}}
\newcommand{\esp}{\end{split}}

\begin{document}

\sloppy
\ninept

\title{Quickest Detection Of Deviations From Periodic Statistical Behavior}

\name{Taposh Banerjee$^{\star}$ \; Prudhvi Gurram$^{\dagger \ddagger}$ \; and \; Gene Whipps$^{\dagger}$ 
\thanks{The work of Taposh Banerjee was supported
by a grant from the Army Research Lab, W911NF1820295.}
}
\address{$^{\star}$ Department of ECE, University of Texas at San Antonio\\
    $^{\dagger}$ U.S. Army Research Laboratory\\
    $^{\ddagger}$ Booz Allen Hamilton\\
    }
%% Create the title:
\maketitle

\vspace{-0.4cm}

\begin{abstract}
A new class of stochastic processes called independent and periodically identically distributed (i.p.i.d.) processes is defined
to capture periodically varying statistical behavior. Algorithms are proposed to detect changes in such i.p.i.d. processes. 
It is shown that the algorithms can be computed recursively and are asymptotically optimal. This problem has applications
in anomaly detection in traffic data, social network data, and neural data, where periodic statistical behavior has been observed. 
\end{abstract}


\vspace{-0.5cm}

\section{Introduction}
In the classical problem of quickest change detection \cite{poor-hadj-qcd-book-2009}, \cite{tart-niki-bass-2014}, \cite{veer-bane-elsevierbook-2013}, 
a decision maker observes a stochastic process with a given distribution. At some point in time, the distribution of the process changes. The 
problem objective is to detect this change in distribution as quickly as possible, with minimum possible delay, subject to a constraint 
on the rate of false alarms. This problem has applications in statistical process control \cite{taga-jqt-1998}, sensor networks \cite{bane-tsp-2015}, cyber-physical system monitoring \cite{chen-bane-tps-2016}, regime changes in neural data \cite{banerjee2018sequential}, traffic monitoring \cite{bane-fusion-2018}, and in general, anomaly detection \cite{bane-fusion-2018}, \cite{bane-globalsip-2018}. 

In many applications of anomaly detection, the observed process has a periodic or regular statistical behavior. Such a periodic statistical behavior 
was observed by us in the multimodal data collected around the Tunnel To Towers 5K run in NYC \cite{bane-fusion-2018}, \cite{bane-globalsip-2018}. 
In these papers, our objective was to detect the 5K run using multimodal data from CCTV cameras, Twitter, and Instagram posts. 
The details on the data collected can be found in these papers. In Fig.~\ref{fig:personInsta}, we have plotted the average counts of the number of persons, 
the number of Instagram posts, and the number of vehicles captured through and around two CCTV cameras in NYC. 
One camera was off the path of the 5K run and another camera was on the path of the run. The data corresponding 
to the off-path camera represents normal behavior. The object counts from CCTV images were extracted using a convolution neural network-based object detector. 
The data for persons and Instagram are from the off-path camera (the Instagram data is collected in a square grid around a camera), and 
the vehicle data is from the on-path camera. As can be seen from the figure, the average counts across different data collection days (here four Sundays) show a similar pattern (growth and decay).  Such periodic or cyclostationary behavior of the data can also be observed in neural spike data \cite{zhang-demba-2018}, \cite{banerjee2018sequential}.
In a controlled experiment, where an animal is trained to do a certain task repeatedly, one can expect a similarity in neural firing patterns \cite{banerjee2018sequential}. 
\begin{figure}
\centering
\includegraphics[scale=0.28]{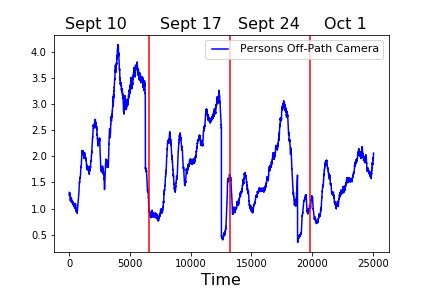}
\hspace{-0.5cm}
\includegraphics[scale=0.28]{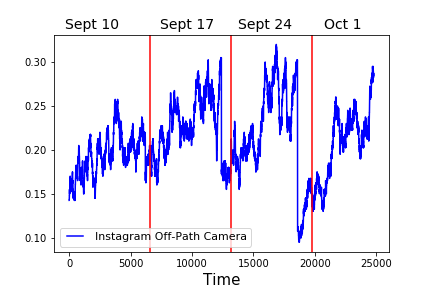}
\includegraphics[scale=0.28]{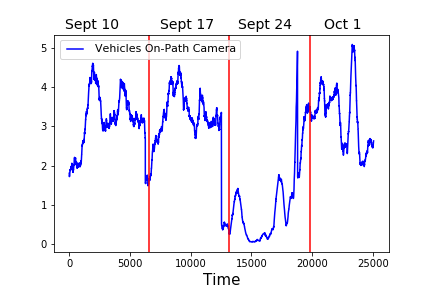}
\caption{The average person, vehicle, and Instagram post counts for data collected in NYC in \cite{bane-fusion-2018}. The figure shows that the average counts 
have similar statistical properties across different days. The vehicle data is from a CCTV camera on the path of the event and captures a decrease in the average counts
on the event day, Sept. 24.}
\label{fig:personInsta}\vspace{-0.4cm}
\end{figure}

The anomaly detection problem in such applications can be posed as the problem of detecting deviations from this regular or periodic statistical behavior. 
In this paper, we develop theory and algorithms to solve this change detection problem. Precise problem formulations are given below. 
The quickest change detection literature is divided broadly into two parts: results for i.i.d. processes with algorithms that can be computed recursively and enjoy strongly optimality properties \cite{mous-astat-1986}, and results for non-i.i.d. data with algorithms that are hard to compute but are asymptotically optimal \cite{lai-ieeetit-1998}, \cite{tart-veer-siamtpa-2005}, \cite{tartakovsky2017asymptotic}, \cite{Pergamenchtchikov2018}. We show in this paper that the algorithms for our non-i.i.d. setup can be computed recursively and are asymptotically 
optimal. A class of models of this type was first studied by us in \cite{bane-globalsip-2018}. In this paper, we study a much broader class of processes and also develop optimality theory.

\vspace{-0.3cm}

\section{Model to Capture Periodic Statistical Behavior}
\vspace{-0.2cm}

An independent and identically distributed (i.i.d.) process is a sequence of random variables that 
are independent and have the same distribution. We define a new category of stochastic processes called 
independent and periodically identically distributed (i.p.i.d.) processes:

\begin{definition}
Let $\{X_n\}$ be a sequence of random variables such that the variable $X_n$ has density $f_n$. The stochastic process $\{X_n\}$
is called independent and periodically identically distributed (i.p.i.d) if $X_n$ are independent and there is a positive integer $T$ such 
that the sequence of densities $\{f_n\}$ is periodic with period $T$:
$$
f_{n+T} = f_n, \quad \forall n \geq 1. 
$$
We say that the process is i.p.i.d. with the law $(f_1, \cdots, f_T)$. 
\end{definition}

Note that the law of an i.p.i.d. process is completely characterized by the finite-dimensional product distribution involving $(f_1, \cdots, f_T)$. 
We assume that in a normal regime, the data can be modeled as an i.p.i.d. process. At some point in time, due to an anomaly, 
the distribution of the i.p.i.d. process deviates from $(f_1, \cdots, f_T)$. Our objective in this paper is to develop algorithms
that can observe the process $\{X_n\}$ in real time and detect changes in the distribution as quickly as possible, subject to a constraint 
on the rate of false alarms. In Section~\ref{sec:Generalipid}, we define the change point model and develop algorithms and optimality theory for detecting changes in 
an i.p.i.d. processes. In Section~\ref{sec:UnknownPost}, we extend the results to the case when the post-change distribution is unknown. In Section~\ref{sec:Paramipid}, we comment on parametric i.p.i.d. models that are easier to learn as compared to learning $(f_1, \cdots, f_T)$.

\section{A Change Detection Theory for General i.p.i.d. Processes}\label{sec:Generalipid}
Consider another periodic sequence of densities $\{g_n\}$ such that 
$$
g_{n+T} = g_{n}, \quad \forall n \geq 1. 
$$
Thus, we essentially have $T$ distinct set of densities $(g_1, \cdots, g_T)$. We assume that at some point in time $\nu$, 
called the change point in the following, the law of the i.p.i.d. process is governed not by the densities $(f_1, \cdots, f_T)$, 
but by the new set of densities $(g_1, \cdots, g_T)$ (a precise definition is given below). 
These densities need not be all different from the set of densities $(f_1, \cdots, f_T)$, 
but we assume that there exists at least an $i$ such that they are different:
\begin{equation}\label{eq:diffpdfassum}
g_i \neq f_i, \quad \text{for some } i = 1, 2, \cdots, T. 
\end{equation}
The change point model is as follows. At a time point $\nu$, the distribution of the random variable changes from $\{f_n\}$ to $\{g_n\}$:
\begin{equation}\label{eq:changepointmodel}
X_n \sim 
\begin{cases}
f_n, &\quad \forall n < \nu, \\
g_n &\quad \forall n \geq \nu.
\end{cases}
\end{equation}
We emphasize that the densities $\{f_n\}$ and $\{g_n\}$ are periodic. This model is equivalent to saying that we have two i.p.i.d. processes, 
one governed by the densities $(f_1, \cdots, f_T)$ and another governed by the densities $(g_1, \cdots, g_T)$, and at the change point $\nu$, 
the process switches from one i.p.i.d. process to another. A more general change point model where the exact post-change density is unknown will be discussed in Section~\ref{sec:UnknownPost}. 

We want to detect the change described in \eqref{eq:changepointmodel} as quickly as possible, subject to a constraint on the rate of false alarms. 
We are looking for a stopping time $\tau$ for the process to minimize a metric on the delay $\tau - \nu$ and to avoid the event of false alarm $\{\tau < \nu\}$. Specifically, 
we are interested in the popular false alarm and delay metrics of Pollak \cite{poll-astat-1985} and Lorden~\cite{lord-amstat-1971}. Let $\Prob_\nu$ denote the
probability law of the process $\{X_n\}$ when the change occurs at time $\nu$ and let $\Expect_\nu$ denote the corresponding expectation. When there is no change, 
we use the notation $\Expect_\infty$. The quickest change detection problem formulation of Pollak \cite{poll-astat-1985} is defined as 
\begin{equation}\label{eq:Pollak}
\begin{split}
\min_\tau &\;\;\; \sup_{\nu\geq 1} \; \Expect_\nu [\tau - \nu | \tau \geq \nu] \\
\text{subj. to}& \;\;\;\; \Expect_\infty[\tau] \geq \beta,
\end{split}
\end{equation}
where $\beta$ is a given constraint on the mean time to false alarm. Thus, the objective is to find a stopping time $\tau$ 
that minimizes the worst case conditional average detection delay subject to a constraint on the mean time to false alarm. 
A popular alternative is the worst-worst case delay metric of Lorden \cite{lord-amstat-1971}: 
\begin{equation}\label{eq:Lorden}
\begin{split}
\min_\tau &\;\;\;\; \sup_{\nu\geq 1} \; \text{ess} \sup \Expect_\nu [\tau - \nu | X_1, \cdots, X_{\nu-1}]\\
\text{subj. to}& \;\;\; \;\Expect_\infty[\tau] \geq \beta,
\end{split}
\end{equation}
where $\text{ess} \sup$ is used to denote the supremum of the random variable $\Expect_\nu [\tau - \nu | X_1, \cdots, X_{\nu-1}]$ outside a set of measure zero. 
Further motivation and comparison of these and other problem formulations for change point detection can be found in the literature \cite{veer-bane-elsevierbook-2013}, \cite{tart-niki-bass-2014}, \cite{poor-hadj-qcd-book-2009}, \cite{lai-ieeetit-1998}.

We now propose a CUSUM-type scheme to detect the above change (also see \cite{lai-ieeetit-1998}). We compute the sequence of statistics 
\begin{equation}\label{eq:PeriodicCUSUM}
W_{n+1} = \max_{1 \leq k \leq n+1} \sum_{i=k}^{n+1} \log \frac{g_{i}(X_{i})}{f_{i}(X_{i})}
\end{equation}
and raise an alarm as soon as the statistic is above a threshold $A$:
\begin{equation}\label{eq:PeriodicCUSUMstop}
\tau_c = \inf \{n \geq 1: W_n > A\}.
\end{equation}
We show below that this scheme is asymptotically optimal in a well-defined sense. But, before that we prove an important property 
that the statistic $W_n$ can be computed recursively and using finite memory. The proof of this and all the other results are provided in Section~~\ref{sec:proofs}. 
\begin{lemma}\label{thm:recurs}
The statistic sequence $\{W_n\}$ can be recursively computed as 
\begin{equation}\label{eq:PeriodicCUSUMrecur}
W_{n+1} = W_n^{+} + \log \frac{g_{n+1}(X_{n+1})}{f_{n+1}(X_{n+1})},
\end{equation}
where $(x)^+ = \max\{x, 0\}$. 
Further, since the set of pre- and post-change densities $(f_1, \cdots, f_T)$ and $(g_1, \cdots, g_T)$ are finite, 
the recursion \eqref{eq:PeriodicCUSUMrecur} can be computed using finite memory needed to store these $2T$ densities. 
\end{lemma}
In the rest of the paper, we refer to \eqref{eq:PeriodicCUSUMrecur} to as the Periodic-CUSUM algorithm. 

Towards proving the optimality of the Periodic-CUSUM scheme, we obtain a universal lower bound on the performance of any stopping time 
for detecting changes in i.p.i.d. processes. Define 
\begin{equation}\label{eq:KLnumber}
I = \frac{1}{T}\sum_{i=1}^T D(g_i \; \| \; f_i),
\end{equation}
where $D(g_i \; \| \; f_i)$ is the Kullback-Leibler divergence between the densities $g_i$ and $f_i$. We assume 
that
$$
D(g_i \; \| \; f_i) < \infty, \quad \forall i=1, 2, \cdots, T,
$$
and
$$
0 < D(g_i \; \| \; f_i), \quad \text{ for some } i=1, 2, \cdots, T.
$$

\begin{theorem}\label{thm:LB}
Let the information number $I$ as defined in \eqref{eq:KLnumber} satisfy $0 < I < \infty$. Then, for any stopping time $\tau$ satisfying the false alarm constraint $\Expect_\infty[\tau] \geq \beta$, we have as $\beta \to \infty$
\begin{equation}\label{eq:LB}
\begin{split}
\sup_{\nu\geq 1} \; \text{ess} &\sup \Expect_{\nu} [\tau - \nu | X_1, \cdots, X_{\nu-1}] \\
& \geq  \sup_{\nu\geq 1} \; \Expect_{\nu} [\tau - \nu | \tau \geq \nu]  \; \geq \; \frac{\log \beta}{I} (1 + o(1)),
\end{split}
\end{equation}
where an $o(1)$ term is one that goes to zero in the limit as $\beta \to \infty$. 
\end{theorem}

We now show that the Periodic-CUSUM scheme \eqref{eq:PeriodicCUSUM}--\eqref{eq:PeriodicCUSUMrecur} is asymptotically optimal for 
both the formulations \eqref{eq:Pollak} and \eqref{eq:Lorden}. 
\begin{theorem}\label{thm:UB}
Let the information number $I$ as defined in \eqref{eq:KLnumber} satisfy $0 < I < \infty$. Then, the  Periodic-CUSUM stopping time $\tau_c$ \eqref{eq:PeriodicCUSUM}--\eqref{eq:PeriodicCUSUMrecur} with $A=\log \beta$ satisfies the false alarm constraint
$$\Expect_\infty[\tau_c] \geq \beta,
$$ 
and as $\beta \to \infty$,
\begin{equation}\label{eq:UB}
\begin{split}
\sup_{\nu\geq 1} \; & \Expect_\nu [\tau_c - \nu | \tau_c \geq \nu]  \\
&\leq \sup_{\nu\geq 1} \; \text{ess} \sup \Expect_\nu [\tau_c - \nu | X_1, \cdots, X_{\nu-1}]  \\
&\leq  \;\; \frac{A}{I} (1 + o(1)) = \;\; \frac{\log \beta}{I} (1 + o(1)).
\end{split}
\end{equation}
%where the information number $I$ is defined in \eqref{eq:KLnumber}. 
\end{theorem}

We note that the algorithm is also optimal for various other formulations studied in the literature \cite{lai-ieeetit-1998}. We do not report these here due to a paucity of space. 

\section{Change Detection With Unknown Post-Change i.p.i.d. Process}\label{sec:UnknownPost}
In the previous section, we assumed that the post-change law $(g_1, \cdots, g_T)$ is known to the decision maker. This information 
was used to design the Periodic-CUSUM algorithm \eqref{eq:PeriodicCUSUMrecur}. In practice, this information may not be available. We 
now show that if the post-change law belongs to a finite set of $M$ possible distributions,
$(g_1^{(1)}, \cdots, g_T^{(1)}), \cdots, (g_1^{(M)}, \cdots, g_T^{(M)})$, then an asymptotically optimal test can be designed. 

For $\ell \in \{1, \cdots, M\}$, define the statistic
\begin{equation}\label{eq:PeriodicCUSUMrecurell}
W_{n+1}^{(\ell)} = \left(W_n^{(\ell)}\right)^{+} + \log \frac{g_{n+1}^{(\ell)}(X_{n+1})}{f_{n+1}(X_{n+1})},
\end{equation}
and the stopping rule
\begin{equation}
\tau_{c\ell} = \inf \left\{n \geq 1: W_{n}^{(\ell)}  \geq \log (\beta M)\right\},
\end{equation}
which is the Periodic-CUSUM stopping rule for the $\ell$th post-change law $(g_1^{(\ell)}, \cdots, g_T^{(\ell)})$.
Now, define
\begin{equation}
\tau_{cm} = \inf \left\{n \geq 1: \max_{1 \leq \ell \leq M} W_{n}^{(\ell)}  \geq \log (\beta M)\right\}.
\end{equation}
Then, note that
\begin{equation}\label{eq:UBmax}
\tau_{cm} \leq \tau_{c\ell}, \quad \forall \; \ell =1, \cdots, M.
\end{equation}
The stopping rule $\tau_{cm}$ is the stopping rule under which we stop the first time any of the $\ell$ Periodic-CUSUMs raise an alarm.

We now show that this stopping rule is optimal for both Lorden's and Pollak's criteria. Towards this end, we define a
Shiryaev-Roberts-type statistic
\begin{equation}
R_n = \sum_{\ell=1}^M \sum_{k=1}^n \prod_{i=k}^n \frac{g_{i}^{(\ell)}(X_{i})}{f_{i}(X_{i})}
\end{equation}
and a Shiryaev-Roberts-type stopping rule
\begin{equation}
\tau_{sr} = \inf \left\{n \geq 1: R_n \geq \beta M\right\}.
\end{equation}
Note that 
\begin{equation}\label{eq:CUSUMleqSR}
\tau_{sr} \leq \tau_{cm}.
\end{equation}
We have the following theorem.
\begin{theorem}\label{thm:unknownpost}
The process $\{R_n - nM\}$ is a $\Prob_\infty$ martingale. If 
$
\Expect_\infty[R_{\tau_{sr}}] < \infty,
$
then
$$
\Expect_\infty[\tau_{cm}] \geq \Expect_\infty[\tau_{sr}] \geq \beta. 
$$
Further, if $(g_1^{(\ell)}, \cdots, g_T^{(\ell)})$ is the true post-change i.p.i.d. law and 
\begin{equation}\label{eq:KLnumberell}
I_\ell = \frac{1}{T}\sum_{i=1}^T D(g_i^{(\ell)} \; \| \; f_i),
\end{equation}
then
\begin{equation}\label{eq:UBMax}
\begin{split}
\sup_{\nu\geq 1} \; & \Expect_\nu [\tau_{cm} - \nu | \tau_{cm} \geq \nu]  \\
&\leq \sup_{\nu\geq 1} \; \text{ess} \sup \Expect_\nu [\tau_{cm} - \nu | X_1, \cdots, X_{\nu-1}]  \\
&\leq  \;\; \frac{\log \beta}{I_\ell} (1 + o(1)).
\end{split}
\end{equation}
Given the lower bound in Theorem~\ref{thm:LB}, the stopping rule $\tau_{cm}$ is thus asymptotically optimal with respect to the criteria of Lorden and Pollak, 
uniformly over each possible post-change hypothesis $(g_1^{(\ell)}, \cdots, g_T^{(\ell)})$, $\ell = 1, \cdots, M$.
\end{theorem}
The condition $\Expect_\infty[R_{\tau_{sr}}] < \infty$ is equivalent to saying that the mean overshoot is finite. This assumption is satisfied for example 
if the likelihood ratios are bounded or may be satisfied if the LLRs have finite variance. The latter will be verified in a future version of this paper. 
Note that our statistics are not random walks (but periodic versions of them). As a result, we cannot directly borrow such finiteness results 
from \cite{wood-nonlin-ren-th-book-1982}, for example.

%\vspace{-1cm}
\section{Detection in Parametric i.p.i.d. Models}\label{sec:Paramipid}
%\vspace{-0.1cm}

In practice, learning pre- and post-change laws $(f_1, \cdots, f_T)$ and $(g_1, \cdots, g_T)$ can be hard. Thus, it is of interest to study low-dimensional parametric i.p.i.d. models. 
Such parametric models were the object of our study in \cite{bane-globalsip-2018} where we assumed that we have a periodic function of parameters $\{\theta_k\}$ with period $T$,
and $X_n \sim f(\cdot, \theta_n)$. Another option is to assume that we have a smooth function 
$
\theta(t), t \in [0,1],
$
and 
\begin{equation}\label{eq:paraipid}
X_n \sim f(\cdot, \theta(n/T)) \quad n = \{1, \cdots, T\}.
\end{equation}
The batch parameter model studied in \cite{bane-globalsip-2018} is then equivalent to a step approximation to $\theta(t)$ in \eqref{eq:paraipid}. 
The change detection problem in this process will be equivalent to detecting a change in the parametric function from $\theta(t)$ to some $\lambda(t)$.
The Periodic-CUSUM algorithm can be easily applied to such models, and all the optimality results proved here are valid for the parametric models as well. We refer the readers to \cite{bane-globalsip-2018} for numerical results and application of our algorithms to NYC data. We do not reproduce them here due to a paucity of space. 

\vspace{-0.3cm}
\section{Conclusions and Future Work}
\vspace{-0.2cm}
We developed a general asymptotic theory for quickest detection of changes in i.p.i.d. models. We also studied the case where the post-change i.p.i.d. law is unknown. 
In future, we will apply the developed algorithm to real multi-modal data, e.g., as collected in \cite{bane-fusion-2018} and \cite{bane-globalsip-2018}. We will also study 
optimality theory for more general change point models in the i.p.i.d. setting. 

%In Fig.~\ref{fig:vehicledata}, we have applied the Periodic-CUSUM algorithm to vehicle count data collected in NYC. The data was collected on four Saturdays, and the event happened pn the weekend of September 24, 2017. Note that the average count plot (left and top) shows periodic behavior of vehicle counts on the non-event days. 
%
%\begin{figure}
%\centering
%\includegraphics[scale=0.28]{Figs/ICASSP_Vehicle_On_Path.png}
%\hspace{-0.5cm}
%\includegraphics[scale=0.28]{Figs/ICASSP_VehicleCount_On_Path.png}
%\includegraphics[scale=0.3]{Figs/ICASSP_Wn_On_Path.png}
%\caption{Application of the Periodic-CUSUM algorithm to vehicle count data from \cite{bane-fusion-2018} and  \cite{bane-globalsip-2018}. The statistic $W_n$ \eqref{eq:PeriodicCUSUMrecur} sharply increases on the event day. }
%\label{fig:vehicledata}
%\end{figure}
%

%\vspace{-1cm}

%\vspace{-0.4cm}
\section{Proofs}\label{sec:proofs}
%\vspace{-0.2cm}

\begin{proof}[Proof of Lemma~\ref{thm:recurs}] For any sequence $\{Z_i\}$ of random variables, we can write
\begin{equation}
\begin{split}
\max_{1 \leq k \leq n+1} & \sum_{i=k}^{n+1} Z_i  = \max\left\{ \max_{1 \leq k \leq n} \sum_{i=k}^{n+1} Z_i, \; Z_{n+1}\right\}\\
=&\max\left\{ \max_{1 \leq k \leq n} \left(\sum_{i=k}^{n} Z_i + Z_{n+1}\right), \; Z_{n+1}\right\}\\
=&\max\left\{ \max_{1 \leq k \leq n} \left(\sum_{i=k}^{n} Z_i \right) + Z_{n+1}, \; Z_{n+1}\right\}\\
=&\max\left\{ \max_{1 \leq k \leq n} \left(\sum_{i=k}^{n} Z_i \right) , \; 0\right\} + Z_{n+1}.
\end{split}
\end{equation}
Substituting $Z_i = \log \frac{g_{i}(X_{i})}{f_{i}(X_{i})}$ into the above equation we get the desired recursion for $W_n$ in \eqref{eq:PeriodicCUSUM}: 
$$
W_{n+1} = W_n^{+} + \log \frac{g_{n+1}(X_{n+1})}{f_{n+1}(X_{n+1})}.
$$
Note that the increment term $\log \frac{g_{n+1}(X_{n+1})}{f_{n+1}(X_{n+1})}$ is only a function of the current observation $X_{n+1}$. Also, 
since the processes are i.p.i.d. with laws  $(f_1, \cdots, f_T)$ and $(g_1, \cdots, g_T)$, the likelihood ratio functions $ \frac{g_{n}(\cdot)}{f_{n}(\cdot)}$ 
are not all distinct, and there are only $T$ such functions $ \frac{g_{1}(\cdot)}{f_{1}(\cdot)}$ to $ \frac{g_{T}(\cdot)}{f_{T}(\cdot)}$. Thus, 
we need only a finite amount of memory to store the past statistic, current observation, and $2T$ densities to compute this statistic recursively. 
\end{proof}

\begin{proof}[Proof of Theorem~\ref{thm:LB}]
Let $Z_i = \log \frac{g_i(X_i)}{f_i(X_i)}$ be the log likelihood ratio at time $i$. We show that the sequence $\{Z_i\}$ satisfies the following statement:
\begin{equation}\label{eq:thm1_1}
\begin{split}
\lim_{n \to \infty} \sup_{\nu \geq 1} \esssup \; \Prob_\nu & \left(  \max_{t \leq n} \sum_{i=\nu}^{\nu+t} Z_i \geq I(1+\delta)n \big| X_1, \cdots, X_{\nu-1}\right) \\
& = 0, \quad \forall \delta > 0,
\end{split}
\end{equation}
where $I$ is as defined in \eqref{eq:KLnumber}. The lower bound then follows from Theorem 1 in \cite{lai-ieeetit-1998}. 
Towards proving \eqref{eq:thm1_1}, note that as $n \to \infty$
\begin{equation}\label{eq:thm1_2}
\begin{split}
\frac{1}{n}\sum_{i=\nu}^{\nu+n} Z_i \to \Expect_1\left[ \sum_{i=1}^T \log \frac{g_i(X_i)}{f_i(X_i)}\right] = I, \quad \text{a. s.} \; \Prob_\nu, \; \; \forall \nu \geq 1. 
\end{split}
\end{equation}
The above display is true because of the i.p.i.d. nature of the observation processes. This implies that as $n \to \infty$
\begin{equation}\label{eq:thm1_3}
\begin{split}
 \max_{t \leq n} \frac{1}{n}\sum_{i=\nu}^{\nu+t} Z_i \to I, \quad \text{a. s.} \; \Prob_\nu, \; \; \forall \nu \geq 1.
\end{split}
\end{equation}
To show this, note that
\begin{equation}\label{eq:thm1_4}
\begin{split}
 \max_{t \leq n} \frac{1}{n}\sum_{i=\nu}^{\nu+t} Z_i  = \max \left\{ \max_{t \leq n-1} \frac{1}{n}\sum_{i=\nu}^{\nu+t} Z_i, \; \;  \frac{1}{n}\sum_{i=\nu}^{\nu+n} Z_i\right\}.
\end{split}
\end{equation}
For a fixed $\epsilon > 0$, because of \eqref{eq:thm1_2}, the LHS in \eqref{eq:thm1_3} is greater than $I(1-\epsilon)$ for $n$ large enough. Also, let the maximum on the LHS be achieved at a point $k_n$, 
then 
$$
 \max_{t \leq n} \frac{1}{n}\sum_{i=\nu}^{\nu+t} Z_i = \frac{1}{n}\sum_{i=\nu}^{\nu+k_n} Z_i  = \frac{k_n}{n} \frac{1}{k_n}\sum_{i=\nu}^{\nu+k_n} Z_i.
$$  
Now $k_n$ cannot be bounded because of the presence of $n$ in the denominator. This implies $k_n > i$, for any fixed $i$, and $k_n \to \infty$. Thus, $\frac{1}{k_n}\sum_{i=\nu}^{\nu+k_n} Z_i \to I$. Since $k_n/n \leq 1$, we have that the LHS in \eqref{eq:thm1_3} is less than $I(1+\epsilon)$, for $n$ large enough. This proves 
\eqref{eq:thm1_3}. 
To prove \eqref{eq:thm1_1}, note that due to the i.p.i.d. nature of the process
\begin{equation}\label{eq:thm1_5}
\begin{split}
\sup_{\nu \geq 1} \esssup \; \Prob_\nu & \left(  \max_{t \leq n} \sum_{i=\nu}^{\nu+t} Z_i \geq I(1+\delta)n \big| X_1, \cdots, X_{\nu-1}\right) \\
=\sup_{1 \leq \nu \leq T} \; \Prob_\nu & \left(  \frac{1}{n}\max_{t \leq n} \sum_{i=\nu}^{\nu+t} Z_i \geq I(1+\delta) \right) 
\end{split}
\end{equation}
The right hand side goes to zero because of \eqref{eq:thm1_3} and because the maximum on the right hand side in \eqref{eq:thm1_5} is over only finitely many terms. 
\end{proof}

\begin{proof}[Proof of Theorem~\ref{thm:UB}]
Again with $Z_i = \log \frac{g_i(X_i)}{f_i(X_i)}$, we show that the sequence $\{Z_i\}$ satisfies the following statement:
\begin{equation}\label{eq:thm2_1}
\begin{split}
\lim_{n \to \infty} \sup_{k \geq \nu \geq 1} \esssup \; \Prob_\nu & \left(  \frac{1}{n} \sum_{i=k}^{k+n} Z_i \leq I - \delta \big| X_1, \cdots, X_{\nu-1}\right) \\
& = 0, \quad \forall \delta > 0.
\end{split}
\end{equation}
The upper bound then follows from Theorem 4 in \cite{lai-ieeetit-1998}. 
To prove \eqref{eq:thm2_1}, note that due to the i.p.i.d nature of the process we have 
\begin{equation}\label{eq:thm2_2}
\begin{split}
&\sup_{k \geq \nu \geq 1} \esssup \; \Prob_\nu  \left(  \frac{1}{n} \sum_{i=k}^{k+n} Z_i \leq I - \delta \big| X_1, \cdots, X_{\nu-1}\right) \\
& = \sup_{\nu + T \geq k \geq \nu \geq 1} \Prob_\nu  \left(  \frac{1}{n} \sum_{i=k}^{k+n} Z_i \leq I - \delta \right) \\
& = \max_{1 \leq \nu \leq T} \max_{\nu \leq k \leq \nu+T} \Prob_\nu  \left(  \frac{1}{n} \sum_{i=k}^{k+n} Z_i \leq I - \delta \right) \\
\end{split}
\end{equation}
The right hand side of the above equation goes to zero for any $\delta$ because of \eqref{eq:thm1_2} and also because of the finite number of maximizations. 
The false alarm result follows directly from \cite{lai-ieeetit-1998} with $A=\log \beta$ because the likelihood ratios here also form a $\Prob_\infty$ martingale.
\end{proof}

\begin{proof}[Proof of Theorem~\ref{thm:unknownpost}]
That $\{R_n - nM\}$ is a $\Prob_\infty$ martingale can be proved by direct verification. For the false alarm proof, we assume that 
$\Expect_\infty[\tau_{sr}] < \infty$, otherwise the proof is trivial. Since $\Expect_\infty[R_{\tau_{sr}}] < \infty$, we have that 
$R_{\tau_{sr}} - \tau_{sr} M$ is integrable. Further, as $n \to \infty$, 
%\begin{equation}
%\begin{split}
%\int_{\tau_{sr} > n}  |R_{\tau_{sr}} - & \tau_{sr} M| \; dP_\infty \leq \int_{\tau_{sr} > n}  R_{\tau_{sr}} + \tau_{sr} M \; dP_\infty \\
%&\leq \beta M \Prob_\infty(\tau_{sr} > n) + \int_{\tau_{sr} > n}  \tau_{sr} M \; dP_\infty \\
%&\to 0, \quad {as}, \; n \to \infty.
%\end{split}
%\end{equation}
\begin{equation}
\begin{split}
\int_{\tau_{sr} > n} & |R_{\tau_{sr}} -  \tau_{sr} M| \; dP_\infty \leq \int_{\tau_{sr} > n}  R_{\tau_{sr}} + \tau_{sr} M \; dP_\infty \\
&\leq \beta M \Prob_\infty(\tau_{sr} > n) + \int_{\tau_{sr} > n}  \tau_{sr} M \; dP_\infty \; \to \; 0.
\end{split}
\end{equation}
Thus, by the optional sampling theorem \cite{wood-nonlin-ren-th-book-1982} and \eqref{eq:CUSUMleqSR} we have
$$
\Expect_\infty[\tau_{cm}] \geq \Expect_\infty[\tau_{sr}]  = \frac{\Expect_\infty[R_{\tau_{sr}}]}{M} \geq \frac{M\beta}{M} = \beta.
$$
The delay result is true because of \eqref{eq:UBmax}. 
\end{proof}

%\vspace{-0.2cm}
%\footnotesize
%\pagebreak
%\pagebreak
%\newpage
\clearpage
\newpage
%\nocite{*}
\bibliographystyle{ieeetr}

%\bibliographystyle{elsarticle-harv}

%\bibliographystyle{elsarticle-harv}
%\bibliography{QCD_verVV}
%\newpage

\bibliography{QCD_verSubmitted}

\end{document}